\documentclass{amsart}
\usepackage{latexsym}
\usepackage[nohug]{diagrams}
\newtheorem{theorem}{Theorem}[section]
\newtheorem{lemma}[theorem]{Lemma}
\newtheorem{proposition}[theorem]{Proposition}
\newtheorem{corollary}[theorem]{Corollary}
\theoremstyle{definition}                               
\newtheorem{definition}[theorem]{Definition}
\newtheorem*{example*}{Example}

\newtheorem{conjecture}[theorem]{Conjecture}

\theoremstyle{remark} 
\newtheorem{remark}[theorem]{Remark}
\newtheorem*{remark*}{Remark}
\numberwithin{equation}{section}

\DeclareMathOperator{\Out}{\sf Out}
\DeclareMathOperator{\Pic}{\sf Pic}
\DeclareMathOperator{\Inn}{\sf Inn}

\DeclareMathOperator{\Aut}{\sf Aut}

\DeclareMathOperator{\ad}{ad}
\DeclareMathOperator{\tr}{tr}

\DeclareMathOperator{\diag}{diag}
\def\Grad{\mbox{\rm{Gr}}^{\mbox{\scriptsize{\rm{ad}}}}}
\def\Grat{\mbox{\rm{Gr}}^{\mbox{\scriptsize{\rm{rat}}}}}

\def\GRad{\mathbb{G}\mbox{\rm{r}}^{\mbox{\scriptsize{\rm{ad}}}}}
\def\GRat{\mathbb{G}\mbox{\rm{r}}^{\mbox{\scriptsize{\rm{rat}}}}}

\def\Gr{\mbox{\rm{Gr}}}

\def\C{\mathcal{C}}

\def\F{\mathcal{F}}
\def\C{\mathcal{C}}
\def\P{\mathbb{P}}

\def\D{\mathcal{D}}
\def\c{\mathbb{C}}
\def\d{\partial}
\def\I{{{\rm I}}}
\begin{document}
\title{Notes on the vector adelic Grassmannian}
%
%
%
%
%
     \author{George Wilson}
     \address{Mathematical Institute, 24--29 St Giles, Oxford OX1 3LB, UK}
      \email{wilsong@maths.ox.ac.uk}
\date{31 December 2009}
\begin{abstract}
These are miscellaneous private notes about 
the generalization of \cite{W2} to the vector case. Some parts date 
from around 1998 and are fairly reliable; others are in a very rough 
state, and may contain errors.
\end{abstract}
\maketitle

\section{The Calogero-Moser spaces}

Fix two positive integers $\, n \,$ and $\, r \,$.  We start off with 
the complex vector space $\, V \equiv V(n,r) \,$ of quadruples 
$\, (X,Y;v,w) \,$ where $\, X \,$ and $\, Y \,$ are $\, n \times n \,$ 
matrices, $\, v \,$ is $\, n \times r \,$ and $\, w \,$ is 
$\, r \times n \,$.  The trace form $\, \tr AB \, $ gives a nonsingular 
pairing between the spaces of $\, p \times q \,$ and  $\, q \times p \,$ 
matrices; in this way we may consider $\, V \,$ as the cotangent bundle 
of the space of pairs $ \, (X,v) \,$, thinking of $\, (Y,w) \,$ as 
a cotangent vector at $ \, (X,v) \,$.  Thus $\, V \,$ is equipped with 
the (holomorphic) symplectic form
$$
\tr(dY \wedge dX \, + \, dw \wedge dv) \ .
$$
The group $\, GL(n,\c) \,$ acts symplectically on $\, V \,$ by 
$$
g(X,Y;v,w) \,:=\, (gXg^{-1}, \,gYg^{-1};\, gv, \,wg^{-1})\ ;
$$
the moment map $\, \mu\,:\, V \to \mathfrak{gl}(n,\c) \,$  for this action 
is
$$
\mu(X,Y;v,w) \,=\, [X,Y] + vw \ .
$$
The action of $\, GL(n,\c) \,$ on $\, \mu^{-1}(-\I) \,$ is free, 
so we can form the symplectic quotient
$$
\C \equiv \C_n(r) \,:=\, {\mu^{-1}(-\I)} / {GL(n,\c)} \ .
$$
For $\, r=1 \,$ this is the space studied in \cite{W2}.
As in that case, $\, \C \,$ is a smooth irreducible affine algebraic 
variety; in fact it is a hyperk\"ahler variety (see \cite{N}), but that 
will not be important here.

We denote by $\, \C' \,$ the dense open subset of $\, \C \,$ consisting 
of classes of quadruples $\, (X,Y;v,w) \,$ where $\, X \,$ has 
distinct eigenvalues.  Fix an orbit representative with $\, X \,$ 
diagonal, say $\, X = \diag(x_1,\ \ldots,\ x_n) \,$.  Let 
$\, \{v_1,\ \ldots,\ v_n\} \,$ be the rows of $\, v \,$, and let 
$\, \{w_1,\ \ldots,\ w_n\} \,$ be the columns of $\, w \,$.  Equating 
the diagonal entries in the equation $\, [X,\, Y] + vw = -\I \,$ gives
\begin{equation}
\label{ii}
v_i w_i \,=\, -1   \text{\ \ \ for\ } 1\leq i \leq n \ , 
\end{equation}
and equating the nondiagonal entries gives 
\begin{equation}
\label{ij}
Y_{ij} \,=\, - \,\, \frac{v_i w_j}{x_i - x_j} \text{\ \ \ for\ } i\neq j\ .
\end{equation}
So just as in the case $\, r=1 \,$, the diagonal entries $\, \alpha_i \,$ 
of $\, Y \,$ are free, and the off-diagonal part of $\, Y \,$ is 
determined by $\, (X,\, v,\, w) \,$.
For $\, r=1 \,$, we can use the action of the diagonal matrices in 
$\, GL(n,\c) \,$ to normalize all 
the (scalars) $\, v_i = 1 \,$ (and  $\, w_i = -1 \,$). As we see from 
\eqref{ii}, the corresponding statement for $\, r>1 \,$ is  
that each pair $\, (v_i, w_i) \,$ determines a point 
of the hyperk\"ahler variety
$$
A_r \,:=\, \{ (\xi, \eta) \,:\, \xi \eta = -1 \} / \c^{\times} \ ,
$$
where $\, \xi \,$ and $\, \eta \,$ are row and column vectors of 
length $\, r \,$, and $\, \lambda \in \c^{\times} \,$ acts by 
$\, (\xi, \eta) \mapsto  (\lambda \xi, \lambda^{-1} \eta) \,$.  

Now, our orbit representative is unique up to the action of a 
diagonal matrix $\, D := \diag(d_1,\ \ldots,\ d_n) \,$ and then of 
a permutation matrix. The action of $\, D \,$ is given by 
$$
y_{ij} \,\mapsto\, d_i d_j^{-1} y_{ij}\,,\ \ v_i \,\mapsto\, d_i v_i \,,\ 
w_i \,\mapsto\, d_i^{-1} w_i \ 
$$
(the parameters $\, (x_i)\,$ and $\, (\alpha_i)\,$ staying fixed).  
This means that $\, \C' \,$ is the quotient by the symmetric 
group $\, \Sigma_n \,$ of the space
$$
(\c^n \setminus \Delta) \,\times\, \c^n \,\times\, A_r^n
$$
with parameters $\, (x_i,\, \alpha_i,\ q_i) \,$, where we have set 
$\, q_i := (v_i,\, w_i) \text{\ mod\ } \c^{\times} \,$.  Further, 
$\, \Sigma_n \,$ acts by simultaneous permutation of the 
parameters $\, (x_i,\, \alpha_i,\ q_i) \,$.  Gibbons and Hermsen interpret 
this as saying that $\, \C' \,$ is the phase space for a system of $\, n \,$ 
indistinguishable particles (located at the points $\, x_i \,$), each 
having the variety $\, A_r \,$ as space of ``internal degrees of freedom''. 
\begin{remark*}
More important for us will be the similar subspace $\, \C^d \,$ of $\, \C \,$, 
where $\, Y \,$ has distinct eigenvalues; $\, \C' \,$ and $\, \C^d \,$
are interchanged by the bispectral involution \eqref{bispinv}.

\end{remark*}

\section{The Gibbons-Hermsen hierarchy}

On the symplectic space $\, V \,$ of quadruples $\, (X,Y;v,w) \,$ we consider 
the $\, GL(n,\, \c)$-invariant Hamiltonians
\begin{equation}
\label{hams}
J_{k, \,\alpha} \,:=\, \tr Y^k v \alpha w\ , \quad k \geq 0\,,\   
\alpha \in \mathfrak{gl}(r,\ \c)\ . 
\end{equation}
\begin{proposition}
\label{GH}
\rm{(i)}  The equations of motion of the system with 
Hamiltonian \eqref{hams} are
\begin{equation}
\label{equns}
\begin{cases}
\,dX/dt &\!\!=\, Y^{k-1} v \alpha w + Y^{k-2} v \alpha w Y + \ldots + 
v \alpha w Y^{k-1}\\
\,dY/dt &\!\!=\, 0\\
\,dv/dt &\!\!=\, Y^k v \alpha\\
\,dw/dt &\!\!=\, - \, \alpha w Y^k \ .
\end{cases}
\end{equation}
\rm{(ii)}  The Poisson brackets of the Hamiltonians \eqref{hams} 
are given by
$$
\{J_{k, \,\alpha},\, J_{l, \,\beta} \} \,=\, J_{k+l, \,[\alpha, \, \beta]} \ .
$$
\end{proposition}

The right hand side of the first equation in \eqref{equns} is interpreted 
to be $\, v \alpha w \,$ if $\, k=1 \,$, and $0$ if $\, k=0 \,$.  

In most cases it is easy to write down the solution to the equations 
\eqref{equns} (see the next subsection).  It is clear even without 
any calculations that the flows of these equations are {\it complete} 
(that is, they exist for all complex $\, t \,$).  Indeed, since $\, Y \,$ 
is constant, the equations for the entries of $\, v \,$ and  $\, w \,$ 
are linear with constant coefficients, so their solutions are 
combinations of polynomials and exponentials; the entries 
in $\, X \,$ are therefore of the same kind.  As Gibbons and Hermsen point out, 
the commutation relations in (ii) are 
the same as those in the Lie algebra of 
polynomial loops in  $\, \mathfrak{gl}(r) \,$ (let the loop 
$\, z \mapsto \alpha z^k \,$ correspond to $\,J_{k, \,\alpha} \,$).  These 
observations suggest  
that the flows of the Hamiltonians \eqref{hams} should fit 
together to give an action on each $\, \C_n(r) \,$ of the group  
$\, \Gamma \,$ of holomorphic functions from $\, \c \,$ to $\, GL_r(\c) \,$. 
As we shall see 
next, on the subspace $\, \C^d \,$ where $\, Y \,$ has distinct eigenvalues, 
we can write down this action explicitly.

\subsection*{Solution of the equations of motion on $\,  \C^d$}
Let $\, Y = \diag(\lambda_1, \,\ldots,\, \lambda_n \,)$\,: then the equations 
\eqref{equns} for the rows $\, v_i \,$ of $\, v \,$ and the columns $\, w_j \,$ 
of $\, w \,$ read
$$
dv_i/dt \,=\, \lambda_i^k v_i \alpha \ , \quad 
dw_j/dt \,=\, - \alpha w_j \lambda_j^k \ ,
$$
with solution
\begin{equation}
\label{vw}
v_i(t) \,=\, v_i(0) \exp(\lambda_i^k \alpha t)\ ,\quad
w_j(t) \,=\, \exp(-\lambda_j^k \alpha t) w_j(0) \ .
\end{equation}
So the diagonal entries in $\, X \,$ satisfy
$$
dX_{ii}/dt \,=\, k \lambda_i^{k-1} v_i(0) \alpha w_i(0) \ ,
$$
with solution
\begin{equation}
\label{xii}
X_{ii} \,=\, X_{ii}(0) \,+\, k \lambda_i^{k-1} v_i(0) \alpha w_i(0) t \ ,
\end{equation}
while for $\, i \neq j \,$ we have
\begin{align*}
dX_{ij}/dt &\,=\, (\lambda_i^{k-1} + \lambda_i^{k-2} \lambda_j + \ldots + 
\lambda_j^{k-1}) (v \alpha w)_{ij}\\
 &\,=\, \frac{\lambda_i^k -\lambda_j^k}{\lambda_i -\lambda_j}\,v_i(0) \,\alpha 
\,\exp \{(\lambda_i^k -\lambda_j^k) \alpha t \} \,w_j(0)\\
 &\,=\, \frac{d}{dt}\, \big[ v_i(0) \,
\exp\{(\lambda_i^k -\lambda_j^k)\alpha t\} \,w_j(0)\big] / 
(\lambda_i -\lambda_j) \ .
\end{align*}
So the solution is 
\begin{equation}
\label{xij}
X_{ij}(t) \,=\, ({\rm constant}) \,+\, v_i(t) w_j(t) / 
(\lambda_i-\lambda_j) 
\end{equation}
(as we could have seen at once from the fact 
that $\, [X,Y] + vw \,$ is constant).
So far the calculation has been valid on the whole of $\, V \,$; but if 
we are in $\, \C \,$, then 
$\, X_{ij}(0) = v_i(0)w_j(0)/(\lambda_i - \lambda_j) \,$, 
so the constant in \eqref{xij} is zero.  We can reformulate that remark 
as follows.
\begin{proposition}
\label{action}
Let $\, (X,\,Y;\,v,\,w) \in \C^d  \,$, with 
$\, Y = \diag(\lambda_1, \,\ldots,\, \lambda_n \,)$.  If 
$\, \gamma(z) \in \Gamma \,$, define  $\, (X,\,Y;\,v,\,w)\circ \gamma  \,$ 
by the formulae
\begin{equation}
\label{gform}
\begin{cases}
\,(v \circ \gamma)_i &\!\!\!\!=\, v_i \gamma(\lambda_i)\\
\,(w \circ \gamma)_j &\!\!\!\!=\, \gamma(\lambda_j)^{-1} w_j\\
\,(X \circ \gamma)_{ii} &\!\!\!\!=\, 
X_{ii} \,+\, v_i \gamma'(\lambda_i)\gamma(\lambda_i)^{-1} w_i\\
\,(X \circ \gamma)_{ij} &\!\!\!\!=\, 
(v \circ \gamma)_i\,(w \circ \gamma)_j\,(\lambda_i - \lambda_j)^{-1}\\
\ \ Y \circ \gamma &\!\!\!\!=\,\, Y \ .
\end{cases}
\end{equation}
Then this defines a right action of $\, \Gamma \,$ on $\, \C^d \,$, and the 
trajectory of the 1-parameter subgroup $\, \{\exp(\alpha z^k t) \} \,$ 
is the solution to the equations of motion \eqref{equns}.
\end{proposition}

We remark that if $\, Y \,$ is diagonal with repeated eigenvalues, 
the equations \eqref{equns} are still easier to solve.  In section~\ref{ex}  
we shall meet 
the extreme case where $\, Y \,$ is a scalar matrix (in that case 
we have $ \,vw = -\I \,$, so it can happen only if $\, n \leq r \,$).  
If $\, Y = \lambda \I \,$, the equations \eqref{gform} 
for the action of $\, \Gamma \,$ 
simplify to
\begin{equation}
\label{yscalar}
\begin{cases}
\,v \circ \gamma &\!\!\!\!=\, v\, \gamma(\lambda)\\
\,w \circ \gamma &\!\!\!\!=\, \gamma(\lambda)^{-1} w\\
\,X \circ \gamma &\!\!\!\!=\, 
X \,+\, v\, \gamma'(\lambda)\gamma(\lambda)^{-1} w\\
\,Y \circ \gamma &\!\!\!\!=\,\, Y \ .
\end{cases}
\end{equation}

\subsection*{The action of scalars}

We denote by $\, \Gamma_{sc} \subseteq \Gamma \,$ the subgroup of scalar-valued 
maps, that is, those of the form $\, \gamma(z) = e^{p(z)} \I \,$ where 
$\, p \,$ is an entire function.  It is easy to write down how 
$\, \Gamma_{sc} \,$ acts on $\, \C \,$.  Note first that 
on the subspace $\, \mu^{-1}(-\I) \,$ of $\, V \,$, we have
\begin{align*}
\tr Y^kvw &\,=\, \tr \{Y^k(-[X,\,Y] - \I)\}\\
          &\,=\, \tr\{[Y,\,Y^k X] - Y^k\}\\
          &\,=\, -\tr Y^k \ .
\end{align*}
Thus  the Hamiltonians $\, J_{k,\, \I} \,$ and $\, -\tr Y^k \,$
induce the same flows 
on our space $\, \C \,$ (on $\, V \,$ the flows differ 
by the action of $\, g(t) := \exp(-tY^k) \,$). The equations 
of motion are now 
simply $\, dX/dt = -kY^{k-1} \,$, other variables constant, with solution
$$
X(t) \,=\, X(0) - kY^{k-1}t\ , \quad  Y,\,v,\,w \text{ constant}\,.
$$
Note particularly that the 1-parameter subgroup $\, \{\exp(xz) \I \} \,$ 
acts on $\, \C \,$ by 
\begin{equation}
\label{Xx}
X \,\mapsto\, X -x\I \ , \quad  Y,\,v,\,w \text{ constant \,}.
\end{equation}
The matrices $\, X - x\I \,$ for different values of $\, x \,$ are never 
conjugate to each other (because their eigenvalues are different); 
thus this subgroup acts {\it freely} on each space 
$\, \C_n \,$ (for $\, n>0 \,$).

Generalizing slightly, we get the following formula for the action 
of $\, \Gamma_{sc} \,$.
\begin{proposition}
\label{gsc}
Let $\, \gamma(z) := e^{p(z)}\I \in \Gamma_{sc} \,$.  Then the action 
of $\, \gamma \,$ on $\, \C \,$ is given by
\begin{equation}
(X,Y;v,w)\circ \gamma \,=\, (X - p'(Y),Y;v,w)
\end{equation}
just as in the case $\, r=1 \,$.

\end{proposition}
\begin{remark*}
There are other cases where we can write down the solution to the 
equations \eqref{equns} for any $\, Y \,$: for example, note that 
$$
\frac{d}{dt}\, (v \alpha w) \,=\, Y^k v \alpha^2 w \,-\, v \alpha^2 w Y^k 
\,=\, [Y^k, \,v \alpha^2 w] \ ,
$$
so if $\, \alpha^2 = 0 \,$ then 
$\, dX/dt \,$ is again constant.  The case where $\, \alpha \,$ is a 
projection ($ \alpha^2 = \alpha $) is another easy one. 
It is annoying that in general 
there does not seem to be one simple formula for the action of 
$\, \Gamma \,$: maybe we need to check that this action really exists. 
\end{remark*}
\section{Some Grassmannians}
We start off in the spirit of \cite{SW}.  Let $\, H \,$ be the Hilbert 
space of $\, L^2 \,$ functions $\, f : S^1 \to \c^r \,$: we shall regard 
the elements of $\, H \,$ as {\it row} vectors 
$\, f = (f_0,\, \ldots,\, f_{r-1}) \,$ of scalar valued functions. 
Let $\, \Gr \,$ be the restricted Grassmannian associated to the usual 
polarization $\, H = H_+ \oplus H_- \,$.  The group 
$\, \Gamma \,$ of entire loops $\, g : \c \to GL(r,\,\c) \,$ acts 
in the obvious way (on the right) on $\, H \,$ and thence on $\, \Gr \,$. 
If $\, W \in \Gr \,$, we have the {\it Baker function} 
$\, \psi_W(g,\, z) \,$: it is the unique (meromorphic) $\, r \times r \,$ 
matrix valued function on $\, \Gamma \times S^1 \,$ such that
\begin{enumerate}
\item{it has the form 
$\, \psi_W(g,\,z) = (\I \,+\, \sum_1^{\infty} a_i(g)z^{-i})\,g(z)\ $;}
\item{each row of $\,\psi_W(g,\,z)\, $ belongs to $\, W \,$ for each 
$\, g \in \Gamma \,$.}
\end{enumerate}
Thus the $\,i^{th} \,$ row of $\, \psi_W \,$ is just the inverse image 
of the $\,i^{th} \,$ basis vector $\, e_i \in \c^r \,$ under the 
projection $\, Wg^{-1} \to H_+ \,$. The Baker function is singular at 
the points $\, g \in \Gamma \,$ where $\, Wg^{-1} \,$ in not in the 
{\it big cell}, that is, where this projection is not an isomorphism. 
We write $\, \widetilde{\psi}_W(g,\,z) \,$ for the term $\, \I + \ldots \,$ 
in (1) above, so that 
$\, \psi_W(g,\,z) = \widetilde{\psi}_W(g,\,z)g(z) \,$.  Note the 
formula
\begin{equation}
\label{Wgamma}
\widetilde{\psi}_{W\gamma}(g,\,z) \,=\, \widetilde{\psi}_W(g \gamma^{-1},\,z) 
\quad (g,\, \gamma \in \Gamma) \ ,
\end{equation}
which follows directly from the definitions.
If we restrict $\, g \,$ to run over the 1-parameter subgroup 
$\, \{ \exp(xz) \I \} \,$ of $\, \Gamma \,$, we obtain the 
{\it stationary Baker function} 
$$
\psi_W(x,\,z) \,:=\, \widetilde{\psi}_W(\exp(xz),\,z)\,e^{xz} \, .
$$
In contrast to the case $\, r=1 \,$, the stationary Baker function 
may not exist, because the flow $\, W \mapsto W e^{xz} \,$ may stay 
entirely outside the big cell; also, even if it does exist, it may be 
independent of $\, x \,$ (see section~\ref{ex} for examples).

In the above discussion, the unit circle $\, S^1 \,$ could be replaced by a 
larger circle, and we can form the union of the resulting Grassmannians as 
the circles get larger.  Changing notation, from now on we shall denote 
this union by $\, \Gr \,$.

In this paper we are mainly concerned with the much smaller Grassmannian 
$\, \Grat \,$: it consists of the subspaces $\, W \in \c(z)^r \,$ 
such that
\begin{enumerate}
\item{$ p(z)\c[z]^r \subseteq W \subseteq q(z)^{-1}\c[z]^r  
\text{ \ for some polynomials \,} p,\,q \,$ ;}
\item{the codimension of $\, W \,$ in $\, q(z)^{-1}\c[z]^r \,$ is 
$\, r \deg(q) \,$.}
\end{enumerate}
As in \cite{W2} or \cite{BW2}, a suitable closure $\, \overline{W} \,$ 
of $\, W \,$ then belongs to $\, \Gr \,$, and we can recover $\, W \,$ 
as the intersection $\, W = \overline{W} \, \cap \c(z)^r \,$.  In this way 
we may regard $\, \Grat \,$ as a subspace of $\, \Gr \,$; the action of 
$\, \Gamma \,$ on  $\, \Gr \,$ clearly preserves this subspace, because 
we have $\, p(z) H \subseteq \overline{W} \subseteq q(z)^{-1} H \,$. 
For each $\, \lambda \in \c \,$ we have the subspace $\, \Gr_{\lambda} \,$ 
of $\, \Grat \,$ consisting of those $\, W \,$ for the polynomials 
$\, p \,$ and $\, q \,$ can be chosen to be powers of $\, z-\lambda \,$.

It will be convenient to work momentarily with the larger space 
$\, \GRat \,$ of all {\it fat} subspaces of $\, \c(z)^r \,$ (that is, 
subspaces satisfying just the first condition in the definition of 
$\, \Grat \,$).  We denote by $\, \GRad \,$ the set of all 
{\it primary decomposable} subspaces in $\, \GRat \,$
[definition to be discussed elsewhere].
Let $\, W \in \GRat \,$.  We define $\, M_W \subseteq \c(z)[\d_z]^r \,$ 
to consist of all operators $\, D = (D_1, \ldots, D_r) \,$ 
such that\footnote{It would be more proper to write 
$\, \c[z] \star D \subseteq W \,$ since $\, M_W \, $ is just the 
space $\, \D(\c[z], W) \,$ studied in section~\ref{Bisp} below.} 
$\, D.\c[z] \subseteq W \,$ (that is, for each polynomial $\, p(z) \,$ 
the vector $\, (D_1.p, \ldots, D_r.p) \,$ is in $\, W \,$).  Clearly, 
$\, M_W \,$ is a fat right sub-$A$-module of $\, \c(z)[\d_z]^r \,$ 
(we recall again from \cite{BGK} that $ \,M \,$ is {\it fat} if  
$\, p A^r \subseteq M \subseteq q^{-1} A^r \,$ for some polynomials 
$\, p \,$ and $\, q \,$).
The following is proved in \cite{CH} for $\, r=1 \,$ and in \cite{BGK} 
in general.
\begin{theorem}
\label{chth}
The map $\, W \mapsto M_W \,$ gives a bijection from $\, \GRad \,$ to 
the set of fat right sub-$A$-modules of $\, \c(z)[\d_z]^r \,$. The inverse 
map sends $\, M \,$ to $\, M.\c[z] \,$.
\end{theorem}

Now let $\, L_W \subseteq \c(z)^r \,$ be the space of all leading 
coefficients of the operators in $\, M_W \,$.  Clearly, $\, L_W \,$ is 
a {\it lattice} in $\, \c(z)^r \,$, that is, a sub-$\c[z]$-module of 
full rank $\, r \,$.
\begin{definition}
The {\it adelic Grassmannian} $\, \Grad \,$ is the subset of those 
$\, W \in \GRad \,$ such that $\, L_W = \c[z]^r \,$.
\end{definition}

This definition (though not the notation) is taken from \cite{BGK}: it 
is probably the most important part of that paper.  In the case 
$\, r=1 \,$, the lattice $\, L_W \,$ is just a (principal) fractional 
ideal $\, (q) \subset \c(z) \,$; the group $\, \c(z)^{\times} \,$ acts 
on $\, \GRad \,$, and we can view $\, \Grad \,$ as the quotient space. For 
$\, r>1 \,$ this point of view is not possible: the group 
$\, GL_r(\c(z)) \,$ acts (compatibly) on $\, \GRad \,$ and on the 
lattices in $\, \c(z)^r $, however, $\, \GRad / GL_r(\c(z)) \,$ is 
not $\, \Grad $, but the quotient $\, \Grad / GL_r(\c[z]) \,$. 

\begin{remark*}
As observed in \cite{BGK}, the larger group $\, GL_r(\c(z)[\d_z]) \,$ 
acts compatibly on fat submodules and on $\, \GRad $. According to Stafford, 
torsion-free $\, A$-modules of rank $\, >1$ are free, so this action is 
transitive.  That leads to an identification
$$
\GRad \,\simeq\, GL_r(\c(z)[\d_z]) / GL_r(A)  \quad (r>1) \ ,
$$
but I do not know what use that is.
\end{remark*}

Another thing to note is that for $\, r=1 \,$ every point of the 
spaces $\, \Gr_{\lambda} \,$ belongs to $\, \Grad $, but for 
$\, r>1 \,$ that is not so.  One reason is provided by the 
following lemma.

\begin{lemma}
\label{zW}
Suppose that $\, W \in \GRad \,$ is such that $\, zW \subseteq W \,$. 
Then $\, W = L_W \,$.
\end{lemma}
\begin{proof}
The condition $\, zW \subseteq W $ implies that 
$\, zM_W \subseteq M_W \,$.  By construction, $\, M_W z \subseteq M_W \,$, 
so $\, M_W \,$ is stable under $\, \ad z \,$. But 
$\, (\ad z)(p \d_z^k) = -k p \d_z^{k-1} \,$; it follows that 
$\, M_W \,$ contains the leading coefficients of all its operators, 
that is, $\, M_W = L_W A  \,$. By Theorem~\ref{chth}, 
$\, W = M_W.\c[z] = L_W \,$.  
\end{proof}
\begin{corollary}
The only point $\, W \in \Grad \,$ such that $\, zW \subseteq W \,$ 
is the base-point $\, \c[z]^r \,$.  
\end{corollary}

\section{The map from $\, \C \,$ to $\,\Grad$}

We fix $\, n \,$ distinct complex numbers $\, \lambda_i \,$, another 
$\, n \,$ complex numbers $\, \alpha_i \,$, and $\, n \,$ pairs 
of vectors $\, (v_i,w_i) \,$ with $\, v_i w_i = -1 \,$ (as
above, the $\, v_i \,$ are $\, 1 \times r \,$ vectors and 
the $\, w_i \,$ are $\, r \times 1 \,$).  To these data we assign the 
following space $\, W \in \Grat(r) \,$: a vector-valued 
rational function $\, f \,$ belongs to $\, W \,$ if
\begin{enumerate}
\item{it is regular everywhere except at the points $\, \lambda_i \,$ 
and at $\, \infty \,$};
\item{it has (at most) a simple pole at each point $\, \lambda_i \,$;}
\item{if $\, f = \sum_{-1}^{\infty} f_k^{(i)} (z- \lambda_i)^k \,$ 
is the Laurent expansion of $\, f \,$ near $\, \lambda_i \,$, then
\begin{enumerate}
\item{$f_{-1}^{(i)}\,$ is a scalar multiple of $\, v_i \,$; and}
\item{($f_0^{(i)} + \alpha_i f_{-1}^{(i)}) w_i = 0 \,$ .}
\end{enumerate}}
\end{enumerate}
Note that in the case $\, r=1 \,$, the condition 3(a) is vacuous, and 
3(b) simply says that 
$f_0^{(i)} + \alpha_i f_{-1}^{(i)} = 0 \,$, so we recover the 
space $\, W(\boldsymbol{\lambda},\,\boldsymbol{\alpha}) \,$ of \cite{W2}.

Let us calculate the Baker function of $\, W \,$.  Conditions (1) and (2) 
above imply that it has the form
$$
\psi_W(g,\,z) \,=\, 
\big( \I \,+\, \sum_1^n \, \frac{A_i(g)}{z-\lambda_i}\big) \,g(z) \ ,
$$
where $\, g \in \Gamma \,$, and the $\, r \times r \,$ 
matrices $\, A_i(g) \,$ are determined by the requirement 
that the rows of $\, \psi_W \,$ all belong to $\, W \,$.
In the Laurent expansion of $\, \psi_W \,$  around $\, z=\lambda_i \,$, 
the residue is $\, A_i(g) g(\lambda_i) \,$, and the constant term is 
$$ 
A_i(g) g'(\lambda_i) \,+\, \big\{ \I \,+\, 
\sum_{j \neq i} \, \frac{A_j(g)}{\lambda_i - \lambda_j} \big\}\, g(\lambda_i)\ .
$$
So condition 3(a) above says that $\, A_i(g) g(\lambda_i) = a_i(g)v_i \,$
for some $\, r \times 1 \,$ vectors $\, a_i \,$; and the condition 3(b) reads
$$
\big\{ A_i(g) g'(\lambda_i) \,+\,  \big[ \I \,+\, 
\sum_{j \neq i} \, \frac{A_j(g)}{\lambda_i - \lambda_j} \,+\, 
\alpha_i A_i(g)\big] \, g(\lambda_i) \big\} \, w_i \,=\, 0 \ .
$$
Combining these equations and using $\, v_i w_i = -1 \,$, we get
$$
a_i(g) \,\big[ v_i g(\lambda_i)^{-1} g'(\lambda_i)w_i \,-\, 
\alpha_i \big] \,+\, \sum_{j \neq i} \, a_j(g) \, 
\frac{v_j g(\lambda_j)^{-1} g(\lambda_i) w_i}{\lambda_i - \lambda_j} \,\,=\,\, 
-g(\lambda_i) \,w_i \ \ .
$$
If we set
\begin{equation}
\label{vg}
\begin{cases}
\,v_i(g)&\!\!\!\!:=\, v_i \,g(\lambda_i)^{-1}\\
\,w_i(g)&\!\!\!\!:=\, g(\lambda_i)\,w_i 
\end{cases}
\end{equation}
that becomes
$$
a_i(g) \,\big[ v_i\,g(\lambda_i)^{-1} g'(\lambda_i) \, w_i
 \,-\, \alpha_i \big] \,+\, \sum_{j \neq i} \, a_j(g) \,
\frac{v_j(g) w_i(g)}{\lambda_i - \lambda_j} \,\,=\,\,
-w_i(g) \ ;
$$
or, finally,
$$
a(g) \, X(g) \,=\, w(g) \ ,
$$
where $\, a(g) \,$ and $\, w(g) \,$ are the $\, r \times n \,$ matrices 
with columns $\, a_i(g) \,$ and $\, w_i(g) \,$, and $\, X(g) \,$ 
is the $\, n \times n \,$ matrix with entries 
\begin{equation}
\label{Xg}
\begin{split}
X_{ii}(g) &\,=\, \alpha_i \,-\,v_i\,g(\lambda_i)^{-1} g'(\lambda_i) \, w_i\\
X_{ij}(g) &\,=\, \frac{v_i(g) w_j(g)}{\lambda_i - \lambda_j} 
\text{\quad for } i \neq j \ .
\end{split}
\end{equation}
Let $\, Y := \diag(\lambda_1,\, \ldots,\, \lambda_n) \,$; then we have
\begin{align*}
\sum_1^n \, \frac{A_i(g)}{z-\lambda_i} &\,=\, 
\sum_1^n \, a_i(g)v_i(g)\,(z-\lambda_i)^{-1}\\
&\,=\, a(g)\,(z\I -Y)^{-1}\,v(g)\\
&\,=\, w(g)X(g)^{-1}\,(z\I -Y)^{-1}\,v(g)\ ,
\end{align*}
where $\, v(g) \,$ is the $\, n \times r \,$ matrix with rows $\, v_i(g) \,$.
So finally we get the formula
\begin{equation}
\label{psiW}
\psi_W(g,\,z) 
\,=\, \big\{ \I \,+\, w(g)X(g)^{-1}\,(z\I -Y)^{-1} v(g) \big\} \,g(z) \ . 
\end{equation}

It is now clear how to define the map $\, \beta : \C^d \to \Grad \,$. 
Let $\, (X,Y;\,v,w) \in \C^d \,$, with 
$\, Y = \diag(\lambda_1,\, \ldots,\, \lambda_n) \,$ and 
$\, X_{ii} = \alpha_i \,$; then $\, \beta \,$ maps  
$\, (X,Y;\,v,w) \,$ to the space $\, W \in \Grad \,$ specified 
by the conditions (1)--(3) at the start of this section.
\begin{proposition}
This map $\, \beta \,:\, \C^d \to \Grad \,$ is $\, \Gamma$-equivariant.

\end{proposition}
\begin{proof}
The quadruple $\, (X(g),Y;\,v(g),w(g)) \,$ given by the 
formulae \eqref{vg} and \eqref{Xg} is just 
$\, (X,Y;\,v,w) \circ g^{-1} \,$, where $\circ$ refers to the 
action \eqref{gform}.  So it follows from \eqref{psiW} that $\, \beta \,$ 
maps $\, (X,Y;v,w) \circ \gamma \,$ to the space with reduced 
Baker function $\, \widetilde{\psi}_W(g \gamma^{-1}, \,z) \,$; 
by \eqref{Wgamma}, that is $\, W \gamma \,$, as claimed.
\end{proof}

If we restrict $\, g \,$ to run over the 1-parameter subgroup 
$\, \{\exp(xz) \I \} \,$ of $\, \Gamma \,$, the formula \eqref{psiW} 
simplifies to 
\begin{equation}
\label{psi1}
\widetilde{\psi}_W(x,\,z) 
\,=\, \I \,+\, w (x\I + X )^{-1}(z\I -Y)^{-1} v  \ . 
\end{equation}
We use this formula to define $\, \beta : \C \to \Grad \,$ in general; 
that is, $\, \beta \,$ maps $\, (X,Y;v,w) \,$ to the point 
$\, W \in \Grad \,$ whose stationary Baker function is given by 
\eqref{psi1}.  The main result of these notes should be the following.
\begin{theorem}
\label{main}
As for $\, r=1 \,$, the maps $\, \beta \,$ give a $\, \Gamma$-equivariant 
bijection
$$
\bigsqcup_{n\geq 0} \,\C_n(r) \,\to\, \Grad(r) \ .
$$
\end{theorem}

I shall not prove that here: but see Proposition~\ref{latt} below, which 
takes care of the first new point (involving the definition of 
$\, \Grad $).  There are two possible ways to prove Theorem~\ref{main}. 
One is to imitate the arguments of \cite{W2} (Shiota's lemma might 
give trouble).  The other is to show that $\, \beta \,$ is the same 
as the bijection constructed in \cite{BGK} (indirectly, passing through 
bundles over the noncommutative quadric).  For $\, r=1 \,$ the only way 
we know to see that relies on having a group acting transitively on 
each $\, \C_n  \,$, so we would need something like Conjecture~\ref{conj} 
below. Probably both ways are 
worth doing, but neither is really satisfactory. Even for $\, r=1 \,$, 
the definition of $\, \beta \,$ is hard to understand, and has the 
capital defect that it is not obvious that it makes sense locally 
(in $\, z \,$).  And no direct definition at all is known for its inverse. 
A more intelligent approach is needed \ldots

\begin{remark*}
In the case $\, r=1 \,$, the equation \eqref{psi1} can be 
rewritten in the form
\begin{equation}
\label{psi2}
\widetilde{\psi}_W(x,\,z) \,=\, 
\det \big\{ \I \,-\, (z\I - Y)^{-1} (x\I + X)^{-1} \big\} 
\end{equation}
(see \cite{W2}, p.\ 16)
The discrepancies between the formulae \eqref{psi1} and \eqref{psi2} 
and the similar ones in \cite{W2} are 
due partly to changes of notation, partly to errors in \cite{W2}   
(these errors are more difficult to make in the case $\, r>1 \,$). 
To recover the exact formulae of \cite{W2} we have to replace 
$\, (X,\,Y;\,v,\,w) \,$ by $\, (X^t,\,-Y^t;\,w^t,\,v^t) \,$, then change the 
sign of one of $\, v \,$ or $\, w \,$, and restore the notation $\, Z \,$ 
instead of $\, Y \,$.
\end{remark*}
\section{More examples}
\label{ex}

This section gathers together various calculations, some of them  
superfluous.

\subsection*{Some points of $\,  \Gr_0 $}

Here we shall look at some of the points $\, W \in \Gr_0 \,$ which 
satisfy $\, z\, \c[z]^r  \subseteq W \subseteq z^{-1} \c[z]^r \,$.  Such 
$\, W \,$ form a space isomorphic the the 
Grassmannian $\, \Gr(r,2r) \,$ of $r$-dimensional subspaces of $\, \c^{2r} \,$. 
If a typical element of $\, W \,$ has Laurent expansion 
$\, f = f_{-1}z^{-1} + f_0 + O(z) \,$, then the various $\, W \in \Gr(r,2r) \,$ 
can be specified by imposing conditions of the form
$$
f_{-1} A \,+\, f_0 B \,=\, 0 \ ,
$$
where the $\, r \times 2r \,$ matrix $\, (A \,| B) \,$ has full rank 
$\, r \,$. 

First, consider one of the ``big cells'' where $\, B \,$ is invertible; 
we may as well suppose that $\, B \,$ is the identity, so the 
corresponding $\, W \,$ 
consists of functions $\, f \,$ such that $\, f_{-1} A + f_0  = 0 \,$.  
A calculation like the one in the preceding section (but much easier) 
shows that the (reduced) Baker function of $\, W \,$ is
\begin{equation}
\label{psiA}
\widetilde{\psi}_W(g,\,z) \,=\, 
\I \,-\, g(0) \{ g(0)^{-1} g'(0) \,+\, A \}^{-1} z^{-1} g(0)^{-1} \ ;
\end{equation}
in particular, the stationary Baker function is
$$
\widetilde{\psi}_W(x,\,z) \,=\, \I \,-\, (x\I + A)^{-1} z^{-1} \ .
$$
Comparing with the formulae \eqref{yscalar}, \eqref{psiW} and \eqref{psi1}, 
we see that $\, W \,$ is the image under $\, \beta \,$ of the quadruple
$$
(X,Y;v,w) \,=\, (A,0;\I,-\I) \in \C_r(r) \,.
$$

Now let us look at the ``opposite'' big cell where $\, A \,$ is the 
identity, so that $\, W \,$ consists of functions 
$\, f \,$ such that $\, f_{-1} + f_0 B  = 0 \,$. Calculating again as 
above, we find this time that the Baker function is
\begin{equation}
\label{psiB}
\widetilde{\psi}_W(g,\,z) \,=\, 
\I \,-\, g(0) B \,\{ \I \,+\, g(0)^{-1} g'(0)B \}^{-1} z^{-1} g(0)^{-1} \ .
\end{equation}
If $\, B = A^{-1} \,$ is invertible, then of course we recover the 
previous formula \eqref{psiA}.  Let us look at the opposite extreme.  
For $\, B = 0 \,$ we just get the base-point $\, W = \c[z]^{r} \,$, so the 
smallest interesting case is where $\, B \,$ has rank $1$, say
$\, B = ab \,$, where $\, a \,$ and $\, b \,$ are column and row  
vectors of length $\,r\,$. Using the formula 
$$
(\I - ab)^{-1} = \I + ab(1 - ba)^{-1}
$$
for the inverse of an elementary matrix (note that $\, ba \,$ is a scalar) 
we can simplify \eqref{psiB} to get
\begin{equation}
\widetilde{\psi}_W(g,\,z) \,=\, 
\I \,-\, g(0) \,\frac{ab}{1 + \beta(g)} \,z^{-1} g(0)^{-1} \ ,
\end{equation}
where we have set $\, \beta(g) := b \,g(0)^{-1} g'(0) \,a \,$. 

Now observe that the condition $\, f_{-1} + f_0 ab \,=\, 0\,$ 
defining $\, W \,$ 
implies that (i) $\, f_{-1} \,$ is a scalar multiple of $\, b \,$; and 
(ii) $\, f_{-1} a = -f_0 aba = -(ba)f_0 a \,$.  If $\, ba \neq 0 \,$, then 
(ii) can be written as $\, (f_0 + \alpha f_{-1})a = 0 \,$, where 
$\, \alpha := (ba)^{-1} \,$.  Thus $\, W \,$ corresponds 
to the quadruple $\, (X,Y;v,w) = (\alpha, 0; b,-\alpha a)) \in \C_1(r) \,$.  
On the other hand, if $\, ba = 0 \,$, then $\, \beta(e^{xz}) \equiv 0 \,$, 
so the {\it stationary} Baker function of $\, W \,$ is
$$
\widetilde{\psi}_W(x,\,z) \,=\, \I \,-\, ab \,z^{-1} \ .
$$
Notice that it is independent of $\, x \,$!  From the formulae 
\eqref{Xx} and \eqref{psi1} we see that this can never happen if $\, W \,$ 
comes from a point of $\, \C \,$, so these spaces $\, W \,$ provide 
examples of points of $\, \Gr_0 \,$ that do not belong  to $\, \beta(\C) \,$. 
However, we knew that already, because $\, ba = 0 \,$ implies 
$\, zW \subseteq W \,$ (cf.\ Lemma~\ref{zW}). 

\subsection*{A direct calculation of $\, L_W $}
Let us look at the example above in the case $\, a = (1,0)^t \,$,  
$\, b = (0,1) \,$, so that  $\, W \in \Gr_0(2) \,$ is the space 
of vectors $\, f \,$ with Laurent expansion of the form 
$$
f \,=\, (0,c)\,z^{-1} \,+\, (c,d) \,+\, O(z) \ . 
$$
\begin{lemma}
\label{exlatt}
The right $\, A$-module $\, \D(\c[z],\,W) \,$ is (freely) generated 
by $\, (1,z^{-1}) \,$ and $\, (0,1) \,$.
\end{lemma}

It follows that the lattice in $\, \c(z)^2 \,$ associated to $\, W \,$ 
has the same generators, confirming Lemma~\ref{zW}.

To check Lemma~\ref{exlatt}, it is slightly easier to work with the space 
of polynomial functions 
$\, V := zW \,$.  A vector $\, f = (f_0(z),\,f_1(z)) \,$ 
belongs to $\, V \,$ if and only if we have
\begin{equation}
\label{fcond}
f_0(0) = 0 \quad \text{and} \quad f_1(0) = f_0'(0) \ .
\end{equation}
Let $\, D := (E,F) \in \D(\c[z], V) \,$.  Then the first condition in 
\eqref{fcond} is equivalent to $\, (E.p)(0) = 0 \,$ 
for all polynomials $\, p \,$; 
that is, $\, E.\c[z] \subseteq z\,\c[z] \,$, or $\, E \in zA \,$.  Set 
$\, E = zP \,$ (where $\, P \in A \,$).  Then $\, E.p = z(P.p) \,$, whence
$\, (E.p)' = P.p + z(P.p)' \,$, so $\, (E.p)'(0) = (P.p)(0) \,$.  So the 
second  condition in \eqref{fcond} is equivalent to 
$$
(F-P).\c[z] \subseteq z\,\c[z] \ ,
$$
so $\, F-P \in zA \,$, 
say $\, F-P = zQ \,$.  Thus $\, \D(\c[z], V) \,$ consists of all 
operators of the form 
$$
D \,=\, (zP,\,P + zQ) \,=\, (z,1)P \,+\, (0,z)Q 
\text{\quad for some \ } P,Q \in A \ .
$$
In other words, the right $\, A$-module 
$\, \D(\c[z], V) \subseteq A^2 \,$ is generated by $\, (z,1) \,$ 
and $\, (0,z) \,$.  Lemma~\ref{exlatt} follows.
\subsection*{Outside the big cell}

Here we give an example where the {\it stationary} 
Baker function does not exist. 
That happens when the flow $\, W \mapsto W e^{xz} \,$ on $\, \Gr(r) \,$ stays 
outside the big cell for all $\, x \,$.  By Proposition 8.6 in \cite{SW}, 
that is impossible for $\, r=1 \,$.  But if we identify $\, H(r) \,$ with 
$\, H := H(1) \,$ via the ``interleaving Fourier series'' isomorphism 
(see \cite{SW}, p.\ 14), then multiplication by $\, z \,$ on $\, H(r) \,$ 
corresponds to multiplication by $\, z^r \,$ on $\, H \,$, so the 
$x$-flow on  $\, \Gr(r) \,$ corresponds to the $\, r^{th} \,$ KP flow 
$\, W \mapsto W e^{t_r z^r} \,$ on $\, \Gr(1) \,$. Any of these flows 
(for $\, r>1 \,$) can have orbits that do not meet the big cell.

Let us take $\, r=2 \,$.  The simplest example I know where the 
$\, t_2$-flow on $\, \Gr(1) \,$ lies outside the big cell starts at the  point
$\, H_S \,$ where $\, S = \{ -3, -1;2,3,\ldots \} \,$.  The 
$\, \tau$-function is ($-24$ times) the Schur function of the 
partition $\, (3,2) \,$, namely
\begin{equation}
\label{tau}
\tau_S(\boldsymbol{t}) \,=\, t_1^5 - 4t_2t_1^3 - 12t_3t_1^2 + 
(12t_2^2 + 24t_4)t_1 - 24t_2t_3 \ .
\end{equation}
The corresponding space $\, W \in \Gr_0(2) \,$ consists of all 
(vector-valued) functions of the form
$\, f(z) = f_{-2} \,z^{-2} + f_{-1} \,z^{-1} + \ldots \,$ satisfying the 
four conditions
\begin{itemize}
\item{the first entry in $\, f_{-2}\,$ vanishes;}
\item{the first entry in $\, f_{-1} \,$ vanishes;}
\item{both entries in $\,f_0 \,$ vanish.}
\end{itemize}
Let us see what happens if we try to calculate the stationary Baker 
function of this $\, W \,$.  It must have the form
$$
\psi_W(x,z) \,=\, (\I + A(x) z^{-1} + B(x) z^{-2})\,e^{xz} \ .
$$
The coefficients of $\, z^{-2}, z^{-1} \text{ and } 1 \,$  in 
$\, \psi_W(x,z) \, $ are (respectively)
$$
B(x)\,; \quad A(x) + x B(x) \,; \quad \text{and} \quad \I + xA(x) + 
\tfrac{1}{2} x^2 B(x) \ . 
$$
The conditions above for the rows of $\, \psi_W \,$ to belong to $\, W \,$ say 
that the first column of $\, B(x) \,$ vanishes; then that 
that the first column of $\, A(x) \,$ vanishes; and finally 
that the first column of $\, \I \,$ vanishes!
\begin{remark}
Let us recall that (in the case $\, r=1 \,$) the roots of the polynomial  
$\, \tau_W(t_1, 0, 0, \ldots) \,$ are the eigenvalues of the matrix $\, X \,$ 
in the pair $\, (X,Y) \in \C \,$ corresponding to $\, W $, and that   
$\, W \,$ is outside the big cell exactly when $\, X \,$ is singular, 
that is when $\, \tau_W(0, 0, 0, \ldots) = 0 \,$. The $\, k^{th} \,$ KP 
flow just translates the variable $\, t_k \,$ in $\, \tau_W \,$.  In our 
example \eqref{tau}, we have $\, \tau_S(0, t_2,0, \ldots) \equiv 0 \,$; 
that means indeed that the $\, t_2$-flow stays outside the big cell. 
This $\, \tau$-function is quite interesting from another point of view: 
we have $\, \tau_S(t_1,0,t_3,0,\ldots) = t_1^5 -12t_3t_1^2 \, $, so the 
$\, t_3$-flow stays inside the ``bad'' set where 
$\, \tau_W(t_1, 0, 0, \ldots) \,$ has a multiple root.
\end{remark}
\section{Transposing matrices}

The appearance of a transpose in the formula \eqref{psitr} below forces us 
to notice some elementary points, which we review in this section

Let $\, R \,$ be a ring: we denote the multiplication in $\, R \,$ by 
juxtaposition and the opposite multiplication by $\, \star \,$, so that 
we have $\, p \star q = qp \,$.  In what follows, $\, D, E, \ldots \,$ 
will denote matrices with entries in $\, R \,$ (not necessarily square, 
but of sizes such that the products we write down are defined).  The 
transpose of a matrix $\, D \,$ is written $\, D^t \,$.  If $\, R \,$ 
is commutative, we have the formula $\, (DE)^t = E^t D^t \,$: in 
general that becomes $\, (D \star E)^t = E^t D^t \,$, or, equivalently
\begin{equation}
\label{newmult}
D \star E \,=\, (E^t D^t)^t \ .
\end{equation}
Here $\, D \star E \,$ denotes the usual product of matrices, but using the 
$\, \star$-multiplication on the scalar entries. 

If $\, b \,$ is an anti-automorphism of $\, R \,$, then we have a 
formula not using $\, \star \,$, namely
\begin{equation}
\label{bt}
b(DE)^t \,=\, b(E)^t \,b(D)^t \ .
\end{equation}
As a special case, let $\, A \,$ be an invertible (square) matrix 
over $\, R \,$.  Then applying the last rule to the formula 
$\, A A^{-1} = \I \,$, we get $\, b(A^{-1})^t \,b(A)^t = \I \,$, that is 
\begin{equation}
\label{binv}
\big[ b(A)^t \big]^{-1} \,=\, b(A^{-1})^t \ .
\end{equation}

Now let $\, \F \,$ be a left $\, R$-module; in what follows $\, \varphi \,$ 
will denote a matrix (not necessarily square) with entries in $\, \F \,$.  
Matrices over $\, R \,$ act on the left of $\, \varphi \,$ in an obvious 
way. Regarding  $\, \F \,$ as a right module over the ring opposite to 
$\, R \,$, we make them act also on the right, setting  
\begin{equation}
\label{staract}
\varphi \star D \,:=\, (D^t.\varphi^t)^t \ ,
\end{equation}
where \,$.$\, denotes the given left action.  We then have the rule
\begin{equation}
\label{stact}
(\varphi \star D)\star E \,=\, \varphi \star (D \star E) \ .
\end{equation}
\begin{remark*}
In our application $\, R \,$ will be a ring of differential operators, 
$\, \F \,$ a module of differentiable functions.
\end{remark*}

\section{Bispectrality}

\label{Bisp}
If in the formula \eqref{psi1} we interchange $\, x \,$ and $\, z \,$ and 
take the matrix transpose, we get 
\begin{equation}
\label{psitr}
\widetilde{\psi}_W(z,x)^t \,=\, \I \,+\, 
v^t(x\I - Y^t)^{-1}(z\I + X^t)^{-1}w^t \ .
\end{equation}
This has the same form as \eqref{psi1}, but with $\, (X,Y;v,w) \,$ 
replaced by 
\begin{equation}
\label{bispinv}
b(X,Y;v,w) \, := \, -\,(Y^t,X^t;w^t,v^t) \ .
\end{equation}
Clearly, $\, b \,$ is an involution on $\, \C \,$; we use the same symbol 
$\, b \,$ to denote the induced involution on $\, \Grad \,$, so that we 
have the formula
\begin{equation}
\label{psib}
\psi_{b(W)}(x,z) \,=\,  \psi_W(z,x)^t \ .
\end{equation}
Note that the corresponding formula for the operator $\, K_W \,$ is 
\begin{equation}
K_{b(U)} \,=\, b(K_U)^t \ ,
\end{equation}
where on the right $\, b \,$ is the anti-automorphism which interchanges 
$\, x \,$ and $\, \d_x \,$.

Now fix two integers $\, r \,$ and $\, s \,$;  let $\, U \in \Grad(r) \,$ 
and $\, V \in \Grad(s) \,$.  We define $\, \D(U,V) \,$ to be the set of 
all $\, r \times s \,$ matrices $\, D \,$ with entries in 
$\, \c(z)[\d_z] \,$ such that $\, U \star D \subseteq V \,$. 
The next proposition generalizes Proposition 8.2 in \cite{BW2}, 
and is proved exactly as in \cite{BW2}.
\begin{proposition}
\label{3eq}
The following are equivalent.\\
{\rm (i)} $\,D \in \D(U,V)\,$.\\
{\rm (ii)} There is an $\, r \times s \,$ matrix $\, \Theta(x) \,$ with 
entries in $\, \c(x)[\d_x] \,$ such that\\
$ \psi_U(x,z) \star D(z) = \Theta(x).\psi_V(x,z) \,$.\\
{\rm (iii)} The $\, r \times s \,$ matrix operator 
$\, K_U \,b(D)(x) \,K_V^{-1} \,$  is differential.\\
Furthermore, the matrix $\, \Theta \,$ is uniquely determined by 
the formula in {\rm (ii)}, and coincides with the operator 
$\, K_U \,b(D)(x) \,K_V^{-1} \,$ in {\rm (iii)}.
\end{proposition}
\begin{proposition}
\label{bbij}
The map $\, D(z) \mapsto \Theta(z)^t \,$ is a bijection from $\, \D(U,V) \,$ 
to $\, \D(b(V), b(U)) \,$.
\end{proposition}
\begin{proof}
Let $\, D \in \D(U,V) \,$.   Interchanging $\, x \,$ and $\, z \,$ 
in the formula in (ii) above characterizing $\, \Theta \,$, we find that 
it is equivalent to 
$$
\psi_U(z,x) \star D(x) \,=\, \Theta(z).\psi_V(z,x) \ . 
$$
By \eqref{psib}, that is the same as
$$
\psi_{b(U)}(x,z)^t \star D(x) \,=\, \Theta(z).\psi_{b(V)}(x,z)^t \ .
$$
By \eqref{staract},  this is equivalent to 
$$
\big( D(x)^t.\psi_{b(U)}(x,z) \big)^t \,=\, 
\big( \psi_{b(V)}(x,z)\star\Theta(z)^t \big)^t \ ,
$$
or, finally,
$$
\psi_{b(V)}(x,z)\star\Theta(z)^t \,=\, D(x)^t.\psi_{b(U)}(x,z) \ .
$$
By Proposition~\ref{3eq}, this means that 
$\, \Theta(z)^t \in \D(b(V), b(U)) \,$.   Proposition~\ref{bbij} follows 
by symmetry.
\end{proof}
\begin{proof}[Second proof]
Let us check that again, this time using the formula in (iii) of 
Proposition~\ref{3eq}.  Let $\, D \in \D(U,V) \,$, so that the 
operator $\, \Theta := K_U \,b(D) \,K_V^{-1} \,$ is differential.  
Using the rules \eqref{bt} and \eqref{binv}, we can calculate
\begin{equation*}
b(\Theta)^t \,=\, b(K_V^{-1})^t \,D^t\, b(K_U)^t 
\,=\, K_{b(V)}^{-1} \,D^t\, K_{b(U)} \ .
\end{equation*} 
Hence $\, D^t = K_{b(V)} \,b(\Theta)^t \,K_{b(U)}^{-1} \,$.  
Since $\, D^t \,$ is differential, Proposition~\ref{3eq} shows again that 
$\, \Theta^t \in \D(b(V),b(U)) \,$. 
\end{proof}

Now let us set $\, V=U \,$ in the above; we write $\, \D(U) \,$ instead 
of $\, \D(U,U) \,$.  By \eqref{stact},  $\, \D(U) \,$ is an algebra 
with respect to the multiplication \eqref{newmult}.  If $\, D \in \D(U) \,$ 
set $\, B(D) := \Theta^t \,$ (where $\, \Theta \,$ is as above). 
\begin{proposition}
The bijection $\, B : \D(U) \to \D(b(U)) \,$ is a anti-isomorphism 
of algebras.
\end{proposition}
\begin{proof}
Calculate: we have 
\begin{align*}
\psi_U \star (D_1 \star D_2) &\,=\, (\psi_U \star D_1) \star D_2\\
&\,=\,  (\Theta_1.\psi_U) \star D_2\\
&\,=\,  \Theta_1.(\psi_U \star D_2)\\
&\,=\,   \Theta_1\Theta_2.\psi_U \ .
\end{align*}
That means that 
\begin{align*}
B(D_1 \star D_2) &\,=\, (\Theta_1\Theta_2)^t\\
&\,=\,  (B(D_1)^t B(D_2)^t)^t\\
&\,=\,  B(D_2) \star B(D_1)\ , 
\end{align*}
as claimed.
\end{proof}

As another application of Proposition~\ref{3eq}, we prove the following 
fact, which should be a first step towards Theorem~\ref{main}.
\begin{proposition}
\label{latt}
Let $\, W \in \beta(\C) \subset \Grat \,$.  Then the lattice 
$\, L_W \,$ 
in $\, \c(z)^r \,$ is the base-point $\, \c[z]^r \,$.
\end{proposition}
\begin{proof}
Let $\, \D := \D(\c[z],W) \,$.  We have to show that
\begin{enumerate}
\item {the leading coefficients of the operators in $\, \D \,$ are all 
(vectors of) polynomials;}
\item { every polynomial $\, p \,$ occurs as the leading coefficient 
of an operator in $\, \D \,$.} 
\end{enumerate}
Since $\, K_{\c[z]} = \I \,$, Proposition~\ref{3eq} tells us that $\, D \,$ 
belongs to $\, \D \,$ if and only if $\, D \,$ and $\, b(D) K_W^{-1} \,$ 
are both differential.  Since
$$
b(b(D) K_W^{-1})^t \,=\, b(K_W^{-1})^t D^t \,=\, K_{b(W)}^{-1} D^t \ ,
$$
it is equivalent to say that $\, D \,$ belongs to $\, \D \,$ if and only if
it is differential and its transpose has the form
$$
D^t \,=\, K_{b(W)} P \,,
$$
where $\, P \,$ is a {\it polynomial} operator (that is, all its 
coefficients are polynomials).  Note first that the leading coefficient 
of $\,D^t \,$ is the same as that of $\, P \,$, hence is a polynomial, which 
proves (1) above.  To prove (2), we have to find enough polynomial 
operators $\, P \,$ such that $\, K_{b(W)} P \,$ is differential.  
For that, recall that if $\, W = \beta(X,Y;v,w) \,$, then 
$$
K_{b(W)} \,=\, \I \,+\, 
v^t(x\I - Y^t)^{-1}(\d_x \I + X^t)^{-1}w^t \ .
$$
So if $\, g(\lambda) := \det(\lambda \I + X) \,$, then 
$\, K_{b(W)} g(\d_x) \,$ is differential; more generally, for any 
$\, p \in \c[z]^r \,$, $\, K_{b(W)} \,p(x)\, g(\d_x) \,$ is differential. 
This operator has leading coefficient $\, p \,$, and its transpose belongs to 
$\, \D \,$ (since $\, P := p(x)\, g(\d_x) \,$ is a polynomial operator).  
That completes the proof.  
\end{proof}
\section{Groups}

So far we have worked with the group $\, \Gamma \,$ of all holomorphic 
maps $\, \c \to GL_r(\c) \,$.  However, even in the case $\, r=1 \,$ this 
group is larger than we need: in that case we usually work with the 
group of functions of the form $\, \gamma(z) = e^{p(z)} $ where 
$\, p \,$ is a {\it polynomial}, and there is no advantage in letting 
$\, p \,$ be an arbitrary entire function.  That would even be confusing. 
One reason is that if $\, p \,$ is a polynomial, then the action of 
$\, \gamma \,$ on ideals in the Weyl algebra $\, A \,$ comes from the 
automorphism $\, D \mapsto e^{p(z)} D e^{-p(z)} \,$ of $\, A \,$; if 
$\, p \,$ is not a polynomial, this map does not preserve $\, A \,$.
For $\, r>1 \,$, the maps of the form $\, \gamma(z) = \exp(p(z)) \,$ 
(where  $\, p \,$ is a polynomial map $\, \c \to \mathfrak{gl}_r(\c)\, $) 
do not form a group.  Of course, we can form the subgroup of $\, \Gamma \,$ 
which they generate, but it is not clear to me what this is.  Perhaps 
for some purposes a smaller group still  will suffice.

Note first that the subgroup $\, \c\I \subset \Gamma_{sc} \,$ of {\it constant} 
scalar matrices acts trivially on the spaces $\, \C \,$ (see 
Proposition~\ref{gsc}), so it is really the quotient group that 
acts; let us change notation so that from now on $\, \Gamma \,$ denotes 
this quotient group.  We then have the decomposition
$$
\Gamma \,=\, \Gamma_{sc} \times \Gamma_1 \ ,
$$
where $\, \Gamma_1 \,$ is the subgroup of maps $\, \gamma \,$ such that 
$\, \det \gamma \,$ is a scalar constant. Let $\, \Gamma_{sc}^{alg} \,$ 
be the subgroup of $\, \Gamma_{sc} \,$ consisting of 
all maps of the form $\, \gamma(z) = e^{p(z)} \I \,$ 
where $\, p \,$ is a {\it polynomial} (say with zero constant term), and set 
\begin{equation}
\Gamma^{alg} \,:=\, \Gamma_{sc}^{alg} \times PGL_r(\c[z]) \,\subset\, \Gamma \ .
\end{equation}
From an algebraic point of view this seems a quite natural group (see 
the next section); however, it has the disadvantage that it does not 
contain the 
1-parameter subgroups $\, \{ \exp(\alpha z^k t) \} \,$ of $\, \Gamma \,$ 
unless $\, \alpha \,$ is nilpotent. 

As well as $\, \Gamma \,$, there is the ``opposite'' group $\, \Gamma' \,$ 
of symplectomorphisms of each $\, \C_n \,$, obtained by reversing 
the roles of $\, (X;v) \,$ and $\, (Y;w) \,$.  We have 
$\, \Gamma' = b \Gamma b \,$ where $\, b \,$ is the bispectral involution 
\eqref{bispinv} (note that $\, b \,$ is {\it anti}-symplectic).  
We might ask whether $\, \Gamma \,$ and $\, \Gamma' \,$ 
generate the symplectomorphism group, but this kind of question 
seems too difficult.  More accessible should be the 
\begin{conjecture}
\label{conj}
The group generated by $\, \Gamma^{alg} \,$ and 
$\, (\Gamma')^{alg} := b \Gamma^{alg} b\,$ acts transitively 
on $\, \C_n \,$.
\end{conjecture}

For $\, r=1 \,$ this is proved in \cite{BW2}.  For $\, r=2 \,$ it is 
presumably the same as the result proved in \cite{BP}; that is, I guess 
that our group is the same as the group of tame automorphisms of the 
quiver considered in \cite{BP}, but that is not clear to me at present.

It is not very clear where ``the'' group referred to in 
Conjecture~\ref{conj} is supposed to live: perhaps in the automorphism 
group of some quiver?  or in $\, \Aut M_r(A) \,$ (see the next section)?  
For $\, r>1 \,$ it is not clear to me that $\, \Aut M_r(A) \,$ acts on
our spaces (even the fat submodules).

There are natural inclusions $\, \C_n(r) \hookrightarrow \C_n(r+1) \,$ 
given by\footnote{If we interpret $\, \C_n(r) \,$ as a space of 
bundles over a noncommutative $\,\P^2 $, this corresponds to adding a 
trivial line bundle. On the level of $\, \Grad$ it sends $\, W \,$ 
to $\, W \oplus \c[z] \,$.} 
adding a column of zeros to $\, v \,$ and a row of zeros to $\, w \,$: 
these inclusions are compatible with the usual embeddings of  
$\, \Gamma (r) \,$ into $\, \Gamma (r+1) \,$.  So we might hope to prove 
Conjecture~\ref{conj} by  induction on $\, r \,$, or even by direct 
reduction to the case $\, r=1 \,$, as in \cite{BP}. 

\section{$\Gamma$ and automorphisms}

Let $\, A^{an} \,$ momentarily denote the algebra of differential 
operators with entire coefficients. Then for each 
$\, \gamma \in \Gamma \,$, the map $\, D \mapsto \gamma D \gamma^{-1} \,$ 
is an automorphism of $\, M_r(A^{an}) \,$.  In general this map does 
not preserve the subalgebra $\, M_r(A) \,$.  More precisely, we have 
the following.
\begin{proposition}
The subgroup of $\, \Gamma \,$ which preserves $\, M_r(A) \,$ is 
exactly $\, \Gamma^{alg} \,$.
\end{proposition}
\begin{proof}
It is obvious that $\, \Gamma^{alg} \,$ preserves $\, M_r(A) \,$.
Conversely, suppose that $\, \gamma \,$ preserves $\, M_r(A) \,$.  
Since $\, \gamma (\d\I) \gamma^{-1} = \d\I - \gamma' \gamma^{-1} \,$, 
$\,\gamma \,$ must satisfy the condition
\begin{equation}
\label{cond1}
\gamma' \gamma^{-1} \text{\, \ is a polynomial matrix.}
\end{equation}
Also, $\,\gamma \,$ must satisfy the condition
\begin{equation}
\label{cond2}
\gamma e_{rs} \gamma^{-1} \text{\ \, is a polynomial matrix for all }
(r,s) \ ,
\end{equation}
where $\, e_{rs} \,$ is the matrix with $1$ in the place $\, (r,s) \,$ 
and zeros elsewhere.  To see that these conditions force $\, \gamma \,$ 
to belong to $\, \Gamma^{alg} \,$, write it 
in the form $\, \gamma(z) = e^{p(z)} \gamma_1(z) \,$, 
where $\, \det \gamma_1 \equiv 1 \,$.  Then 
$\, \gamma' \gamma^{-1} = p'(z) \I + \gamma_1' \gamma_1^{-1} \,$.  
Here the second term has trace zero, so \eqref{cond1} shows that $\, p' $, 
hence also $\, p \,$, is a polynomial.  It remains to show that 
$\, \gamma_1 \,$ is polynomial.  Clearly, $\, \gamma_1 \,$ satisfies 
the same condition \eqref{cond2} as $\, \gamma \,$.  That this implies 
that it is polynomial 
will follow from the stronger statement: if $\, A \in \Gamma \,$ 
and $\, B \in \Gamma_1 \,$ and all the matrices 
$\, A(z) e_{rs} B(z) \,$ are polynomial, then $\, A \,$ is 
polynomial.  Indeed, the $\, (i,j) \,$ entry in $\, A e_{rs} B \,$ 
is $\, A_{ir} B_{sj} \,$, so our condition says that the product 
of any entry in $\, A \,$ with any entry in $\, B \,$ is a polynomial.
Fix an entry in $\, A \, $, say $\, a(z) \,$; then for all $\, (i,j) \,$ 
we have 
$$
a(z) B_{ij}(z) \,=\, p_{ij}(z) 
$$
where the $\, p_{ij} \,$ are polynomials.  Since $\, \det B \equiv 1 \,$, 
taking determinants gives $\, a(z)^r = \det(p_{ij}) \,$. Thus  
$\, a(z)^r \,$ is a polynomial; since $\, a \,$ is entire, this implies  
that $\, a(z) \,$ is a polynomial, as required.
\end{proof}

In the next proposition we regard $\, \Gamma^{alg} \,$ as a subgroup 
of $\, \Aut M_r(A) \,$. 
\begin{proposition}
$\Gamma^{alg} \,$ is exactly the isotropy group of $\, z\I \in M_r(A) \,$. 
\end{proposition}
\begin{proof}
It is obvious that $\, \Gamma^{alg} \,$ fixes $\, z\I \,$. The converse 
depends on the fact that $\, \Aut M_r(A) \,$ is the semi-direct 
product of $\, \Inn M_r(A) \,$ and $\, \Aut A \,$ (acting on each 
matrix entry); thus every automorphism has the form 
$\, D \mapsto T \sigma(D) T^{-1} \,$ for some $\, \sigma \in \Aut A \,$, 
$\, T \in GL_r(A) \,$.  If this fixes $\, z\I \,$, then we have 
$\, z\I = T\sigma(z)\I T^{-1} \,$, or $\, zT = T \sigma(z) \,$.  
Looking at the leading coefficient of the matrix operator $\, T \,$, 
we conclude that $\, \sigma(z) = z \,$, that is, 
$\, \sigma \in \Gamma_{sc}^{alg} \,$.   We now have 
$\, zT = Tz \,$, which implies that $\, T \,$ is 
an operator of order zero in $\, \d_z \,$, hence belongs to 
$\, GL_r(\c[z]) \,$ as claimed.
\end{proof}

For completeness, we give the proof that $\, \Aut M_r(A) \,$ is a semi-direct 
product, as claimed above.  For any algebra $\, A \,$ we have a 
commutative diagram
\begin{equation*}
\begin{diagram}[small]
&\Aut A & \rTo & \Pic A \\
&\dInto^{\iota} & \rdDotsto & \dTo_{\simeq}\\
\Inn M_r(A) \hookrightarrow \ & \Aut M_r(A)& \rTo& \Pic M_r(A) 
\end{diagram}
\end{equation*}
where the inclusion $\, \iota \,$ makes an automorphism of $\, A \,$ 
act separately on each entry of a matrix. If $\, A \,$ is the Weyl 
algebra, then the top horizontal arrow, and hence also 
the dotted diagonal arrow, is an isomorphism.  We may use this dotted arrow 
to identify $\, \Aut A \,$ with $\, \Out M_r(A) \,$; the vertical 
arrow $\, \iota \,$ then splits the exact sequence in the second row 
of the diagram.
\bibliographystyle{amsalpha}

\begin{thebibliography}{A}
%
\bibitem[BGK]{BGK}
V.~Baranovsky, V.~Ginzburg and A.~Kuznetsov, \textit{Wilson's Grassmannian 
and a noncommutative quadric}, Internat.\ Math.\ Res.\ Notices  
\textbf{21} (2003), 1155--1197.
%
\bibitem[BW2]{BW2}
Yu. Berest and G. Wilson, \textit{Automorphisms and ideals of the 
Weyl algebra}, Math. Ann. \textbf{318} (2000), 127--147.
%
\bibitem[BP]{BP}
R.\ Bielawski and V.\ Pidstrygach, \textit{On the symplectic structure 
of instanton moduli spaces}, {\tt arXiv}:0812.4918.
%
\bibitem[CH]{CH}
R. C. Cannings and M. P. Holland, \textit{Right ideals of rings
of differential operators}, 
J. Algebra \textbf{167} (1994), 116--141.
%
\bibitem[GH]{GH}
J.\ Gibbons and T.\ Hermsen, \textit{A generalisation of the 
Calogero-Moser system}, Physica \textbf{11D} (1984), 337--348.
%
\bibitem[N]{N}
H.\ Nakajima, Lectures on Hilbert schemes of points on surfaces, 
University Lecture Series \textbf{18}, Amer.\ Math.\ Society, 
Providence RI, 1999. 
%
\bibitem[SW]{SW}
G. Segal and G. Wilson,
\textit{Loop groups and equations of KdV type}, 
Publ. Math. IHES \textbf{61} (1985), 
5--65.
%
\bibitem[W2]{W2}
G. Wilson,
\textit{Collisions of Calogero-Moser particles and
an adelic Grassmannian} (with an Appendix by I. G. Macdonald), 
Invent. Math. \textbf{133} 
(1998), 1--41.
%
\end{thebibliography}

\end{document}